%% file: ComputingHomologyInvts.tex
\documentclass[11pt]{amsart}

\usepackage{amsmath, amssymb, amsthm, amsfonts, graphicx, pinlabel}

\input{Sections/Preamble}

\begin{document}

\title[Computing Homology Invariants of Legendrian Knots]{Computing Homology Invariants of Legendrian Knots}

\author{Emily E. Casey}
\address{Siena College, Loudonville, NY 12211}

\author{Michael B. Henry}
\address{Siena College, Loudonville, NY 12211}
\email{mbhenry@siena.edu}

\begin{abstract}

The Chekanov-Eliashberg differential graded algebra of a Legendrian knot $L$ is a rich source of Legendrian knot invariants, as is the theory of generating families. The set $P(L)$ of homology groups of augmentations of the Chekanov-Eliashberg algebra is an invariant, as is a count of objects from the theory of generating families called graded normal rulings. This article gives two results demonstrating the usefulness of computing the homology group of an augmentation using a combinatorial interpretation of a generating family called a Morse complex sequence \cite{Henry2013}. First, we show that if the projection of $L$ to the $xz$-plane has exactly 4 cusps, then $|P(L)| \leq 1$. Second, we show that two augmentations associated to the same graded normal ruling by the many-to-one map between augmentations and graded normal rulings defined by Ng and Sabloff \cite{Ng2006} need not have isomorphic homology groups.
\end{abstract}

\maketitle

\input{Sections/Introduction}
\input{Sections/Background}

\input{Sections/LCH}

\input{Sections/NormalRulings}

\input{Sections/MCS}
\input{Sections/Theorem1}
\input{Sections/Theorem2}

\addcontentsline{toc}{section}{Bibliography} 
\bibliographystyle{amsplain}
\bibliography{Bibliography}

\end{document}

%% file: Sections/Introduction.tex
\section{Introduction}
\mylabel{ch:intro}

The classification of Legendrian knots in the standard contact structure on $\R^3$ has been significantly advanced by invariants derived from the Floer-theoretic techniques of symplectic field theory \cite{Eliashberg, Eliashberg2000} and the classical Morse-theory of generating families \cite{Chekanov2005,Jordan2006,Traynor2001}. Although these two approaches to Legendrian knot theory have different geometric foundations, many connections have been found between the invariants they define. In this article, we give two results that demonstrate the usefulness of a recently defined object, called a Morse complex sequence, in deepening understanding of these connections. 

Chekanov \cite{Chekanov2002a} and, independently, Eliashberg \cite{Eliashberg2000} assign a differential graded algebra to a Legendrian knot $\Leg$ that is, in the case of \cite{Eliashberg2000}, a special case of symplectic field theory. The Chekanov-Eliashberg algebra is a Legendrian invariant, up to an appropriate algebraic equivalence. Ng \cite{Ng2003} gives a description of the Chekanov-Eliashberg algebra of $\Leg$ in terms of the $xz$-projection $\front$, called the front diagram of $\Leg$. Though easy to define, the Chekanov-Eliashberg algebra is difficult to employ as an invariant. However, more manageable Legendrian invariants are defined from certain maps from the Chekanov-Eliashberg algebra to $\Z / 2 \Z$ called augmentations. The set $Aug(\front)$ consists of augmentations of the Chekanov-Eliashberg algebra defined on $\front$. Each augmentation $\aug$ determines a $\Z / 2 \Z$ chain complex and the Poincar\'{e} polynomial of the resulting homology group is called the Chekanov polynomial of $\aug$ and written $P_{\aug}(t)$. Much is known about augmentations and their Chekanov polynomials; for example, Chekanov \cite{Chekanov2002} proves $\{P_{\aug}(t)\}_{\aug\in Aug(\front)}$ is a Legendrian invariant, Sabloff \cite{Sabloff2006} proves that the coefficients of $P_{\aug}(t)$ satisfy a duality relationship, and Melvin and Shrestha \cite{Melvin2005} prove that, for any natural number $n$, there exists a Legendrian knot with $n$ Chekanov polynomials. Theorem~\ref{t:two-cusp} proves that the front diagram of a Legendrian knot with more than one Chekanov polynomial has more than four cusps. 

\begin{reptheorem}{t:two-cusp}
If the front diagram of a Legendrian knot has exactly four cusps, then the Legendrian knot has at most one Chekanov polynomial. 
\end{reptheorem}

A graded normal ruling on $\front$ is a bijection between the left and right cusps of $\front$ along with, for each pair of identified cusps, two paths between those cusps that satisfy certain requirements; see Figure~\ref{f:ruling-pieces} for a selection of those requirements and Figure~\ref{f:ruling} for an example of a graded normal ruling. Chekanov and Pushkar \cite{Chekanov2005} show that a generating family for $\Leg$ defines a graded normal ruling on $\front$ and, in a similar spirit, the second author \cite{Henry2011} shows that a Morse complex sequence defines a graded normal ruling on $\front$. The set of all graded normal rulings on $\front$ is $\mathcal{R}^0(\front)$. Fuchs \cite{Fuchs2003} proves that if $\mathcal{R}^0(\front)$ is non-empty, then $Aug(\front)$ is non-empty as well. Fuchs and Ishkanov \cite{Fuchs2004} and, independently, Sabloff \cite{Sabloff2005} prove the converse. Ng and Sabloff \cite{Ng2006} further clarify the relationship between augmentations and graded normal rulings by proving that there exists an algorithmically defined many-to-one map $\Psi: Aug(\front) \to \mathcal{R}^0(\front)$. Josh Sabloff posed the following question to the second author, ``Does $\Psi(\aug_1) = \Psi(\aug_2)$ imply $P_{\aug_1}(t) = P_{\aug_2}(t)$?'' In other words, are the Chekanov polynomials determined by graded normal rulings? Theorem~\ref{t:infinite-family} answers the question in the negative. 

\begin{reptheorem}{t:infinite-family}
For any natural number $m \geq 2$, there exists a front diagram $\front_m$ with graded normal ruling $\ruling$ and augmentations $\aug_1, \hdots, \aug_m$ so that:
\begin{enumerate}
	\item For all $1 \leq i \leq m$, $\Psi ( \aug_i) = \ruling$; 
	\item If $i \neq j$, then $P_{\aug_i}(t) \neq P_{\aug_j}(t)$; and
	\item The smooth knot type of $\front_m$ is prime.
\end{enumerate}
\end{reptheorem}

Recently, the idea of a Morse complex sequence, originally introduced by Petya Pushkar and first appearing in print in \cite{Henry2011}, has proven to be useful in further refining connections between the Chekanov-Eliashberg algebra and invariants derived from generating families; see \cite{Henry2011, Henry2013, Henry2014}. Informally, a Morse complex sequence, abbreviated MCS and denoted $\MCS$, of $\front$ is a combinatorial/algebraic analogue of a generating family. The set $MCS(\front)$ consists of all MCSs of $\front$. In \cite{Henry2013}, an MCS $\MCS$ is assigned a differential graded algebra that, conjecturally, extends a homological Legendrian invariant derived from generating families to an algebra Legendrian invariant. The homology of the linear level of the MCS algebra of $\MCS$ defines a Poincar\'{e} polynomial called the MCS polynomial of $\MCS$ and written $P_{\MCS}(t)$. By Corollary 7.12 of \cite{Henry2013}, $\{P_{\aug}(t)\}_{\aug \in Aug(\front)} = \{P_{\MCS}(t)\}_{\MCS \in MCS(\front)}$ holds. Therefore, questions concerning the Chekanov polynomials of a Legendrian knot may be framed in terms of MCS polynomials. Beyond the statements of the two main results, this article is meant to demonstrate the usefulness of such a translation, as the proofs of both Theorem~\ref{t:two-cusp} and \ref{t:infinite-family} take this approach. 

\subsection{Outline of the article}

Section~\ref{ch:background} provides the background material in Legendrian knot theory necessary to prove the main results in Section~\ref{ch:results}. Section~\ref{ss:rulings} includes two technical results concerning 2-graded normal rulings, Propositions~\ref{prop:fmax} and \ref{prop:2bridgefmax}, used to prove the smooth knots in Theorem~\ref{t:infinite-family} are prime. Section~\ref{ss:MCSs} gives carefully chosen background material on MCSs and the MCS algebra so as to include what is necessary for Section~\ref{ch:results}, but avoid most of the technical details of MCSs. 

\subsection{Acknowledgments}
The authors would like to thank Josh Sabloff and Dan Rutherford for many fruitful discussions. In particular, Dan suggested the approach to proving primality used in Theorem~\ref{t:infinite-family}. This work was supported by the Siena College Summer Scholars undergraduate research program while the first author was a student at Siena College.

%% file: Sections/Background.tex
\section{Background}
\mylabel{ch:background}

The \textbf{standard contact structure} on $\R^3$ is the $2$-plane distribution $\xi_{std}$ defined by the $1$-form $dz - y dx$. A smooth knot $\Leg : S^1 \to \R^3$ is \textbf{Legendrian} if $\Leg'(t) \in \xi_{std}$ for all $t \in S^1$. Two Legendrian knots $\Leg_0$ and $\Leg_1$ are \textbf{equivalent} if there exists a smooth map $\Leg: S^1 \times [0,1] \to \R^3$, called a \textbf{Legendrian isotopy}, so that $L_0 = L(\cdot,0)$, $L_1=L(\cdot,1)$, and $\Leg(\cdot,t)$ is a Legendrian knot for all $t \in S^1$. The projection of $\Leg$ to the $xz$-plane is the \textbf{front diagram} $\front$ of $\Leg$. Every Legendrian knot is equivalent to a Legendrian knot whose front diagram has transverse double points, called \textbf{crossings}, and semicubical cusps; see Figure~\ref{f:nearly-plat}. A \textbf{strand} of $\front$ is a smooth path in $\front$ with one endpoint at a left cusp and the other at a right cups. A front diagram is \textbf{plat} if all left cusps have the same $x$-coordinate, all right cusps have the same $x$-coordinate, and no two crossings have the same $x$-coordinate. A front diagram is \textbf{nearly plat} if all cusps and crossings have distinct $x$-coordinates and it is equivalent to a plat front diagram by an arbitrarily small Legendrian isotopy; see Figure~\ref{f:nearly-plat}. All Legendrian knots considered in this article have nearly plat front diagrams.

\begin{figure}[t]
\labellist
\small\hair 2pt
\pinlabel {$i$} [tl] at 305 103
\pinlabel {$i+1$} [br] at 315 118
\pinlabel {$i$} [tl] at 410 103
\pinlabel {$i+1$} [bl] at 406 118
\pinlabel {(a)} [tl] at 351 93
\pinlabel {(b)} [tl] at 351 46
\pinlabel {(c)} [tl] at 351 0
\endlabellist
\centering
\includegraphics[scale=.9]{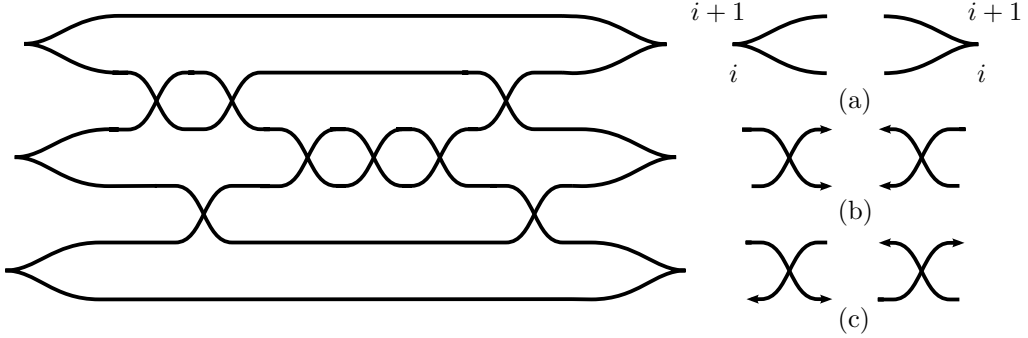}
\caption{(Left) A nearly plat front diagram for a Legendrian knot. (Right) (a) A Maslov potential near cusps of $\front$. (b) Positive crossings. (c) Negative crossings. }
\label{f:nearly-plat}
\end{figure}

We define the two ``classical'' Legendrian knot invariants in terms of the front diagram. Given an oriented Legendrian knot $\Leg$, the \textbf{rotation number} $r(\Leg)$ is $(d-u)/2$ where $d$ (resp. $u$) is the number of cusps in $\front$ at which the orientation travels downward (resp. upward). All Legendrian knots considered in this article have rotation number 0. A crossing of the front diagram $\front$ is \textbf{positive} (resp. \textbf{negative}) if the two crossing strands are both oriented to the left or both oriented to the right (resp. both oriented down or both oriented up); see Figure~\ref{f:nearly-plat} (b) and (c). The \textbf{writhe} $w(\front)$ of $\front$ is the number of positive crossings minus the number of negative crossings. The \textbf{Thurston-Bennequin number} $tb(\Leg)$ is $w(\front)$ minus half the number of cusps. Given a smooth knot $K$, $\overline{tb}(K)$ is the maximum Thurston-Bennequin number over all Legendrian knots smoothly isotopic to $K$.

%% file: Sections/LCH.tex
\subsection{Chekanov Polynomials}
\label{s:LCH}

Suppose $\Leg$ is a Legendrian knot whose front diagram $\front$ has crossings and cusps with distinct $x$-coordinates. We fix a map $\mu: \Z / (2 r(\Leg) \Z) \to \Leg$, called a \textbf{Maslov potential}, that is constant except at cusp points of $\front$, where it changes as in Figure~\ref{f:nearly-plat}~(a). Assign the labels $Q=\{q_1, \hdots, q_n\}$ to the crossings and right cusps of $\front$ from left to right. We define a grading $| \cdot |: Q \to \Z / (2 r(\Leg) \Z)$ by $|q_i| = 1$ if $q_i$ is a right cusp and, otherwise, $|q_i| = \mu(T) - \mu(B)$, where $T$ and $B$ are the strands of $\front$ crossing at $q_i$ and $T$ has smaller slope. Since $\Leg$ has one component, the grading does not depend on the chosen Maslov potential. Let $A(\front)$ be the $\Z / 2 \Z$ vector space freely generated by $Q$ and let $\alg(\front)$ be the unital tensor algebra $TA(L)$ graded by $|\cdot|$. The \textbf{Chekanov-Eliashberg differential graded algebra} of $\front$ is the pair $(\A(\front), \df)$, where $\df: \A(\front) \to \A(\front)$ is a certain degree $-1$ differential. The original formulation of the Chekanov-Eliashberg algebra, in terms of the projection of $\Leg$ to the $xy$-plane, appears in \cite{Chekanov2002a}; \cite{Ng2003} provides a description in terms of the front diagram. We will not define the map $\df$ as there is no need to work with it in this article. 

The Chekanov-Eliashberg algebra is a Legendrian isotopy invariant, up to an algebraic equivalence called stable tame isomorphism. The idea of an augmentation, also first formulated in \cite{Chekanov2002a} in the context of Legendrian knots, provides a method for extracting more easily computed Legendrian isotopy invariants from the Chekanov-Eliashberg algebra. An \textbf{augmentation} is an algebra map $\aug : (\A(\front), \df) \to \Z / 2 \Z$ satisfying $\aug(1)=1$, $\aug \circ \df = 0$, and $\aug(q) =1$ only if $|q| = 0$. We say a crossing $q$ is \textbf{augmented} by $\aug$ if $\aug(q) = 1$. The set $Aug(\front)$ is the set of all augmentations of $(\A(\front), \df)$. 

Given $\aug \in Aug(\front)$, we define $\df^{\aug}$ to be the differential $\phi^{\aug} \circ \df \circ (\phi^{\aug})^{-1}$, where $\phi^{\aug} : \A(\front) \to \A(\front)$ is the algebra homomorphism defined on generators by $\phi^{\aug}(q) = q + \aug(q)$. This differential has the property that $\df^{\aug}_{1} \circ \df^{\aug}_{1} = 0$, where $\df^{\aug}_1(q)$ are the length 1 monomials of $\df^{\aug}$. Thus, $(A(\front), \df^{\aug}_1)$ is a finite-dimensional chain complex called a \textbf{linearization} of the Chekanov-Eliashberg algebra. We let $h^{\aug}_i$ be the dimension in degree $i$ of the homology of $(A(\front), \df^{\aug}_1)$ and define the \textbf{Chekanov polynomial} $P_{\aug}(t)$ to be $$P_{\aug}(t) = \sum_{i \in \Z} h_i^{\aug} t^i.$$ The collection $\{ P_{\aug}(t) \}_{\aug \in Aug(\front)}$ is a Legendrian isotopy invariant \cite{Chekanov2002}. The coefficients of a Chekanov polynomial satisfy a duality relationship.

\begin{theorem}[\cite{Sabloff2006}]
\label{t:duality}
Given $\aug \in Aug(\front)$, 
$$
h^{\aug}_i = \left\{ \begin{array}{rl}
 h^{\aug}_{-i} &\mbox{ if $i \neq 1$} \\
 h^{\aug}_{-1} + 1 &\mbox{ if $i = 1$.}
       \end{array} \right.
$$
\end{theorem}

%% file: Sections/NormalRulings.tex
\subsection{Normal Rulings}
\label{ss:rulings}

Suppose $\front$ is the front diagram of a Legendrian knot $\Leg$, the rotation number of $\Leg$ is 0, and the crossings and cusps of $\front$ have distinct $x$-coordinates. Fix a Maslov potential $\Maslov$ and assign a degree to each crossing as in Section~\ref{s:LCH}. 

\begin{figure}[t]
\centering
\includegraphics[scale=.9]{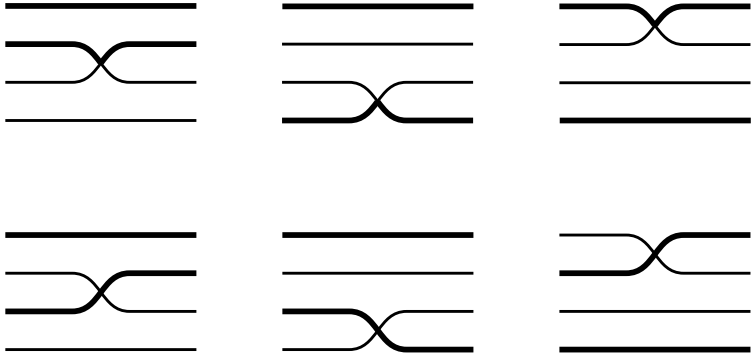}
\caption{The paths of a normal ruling near switches (top row) and returns (bottom row).}
\label{f:ruling-pieces}
\end{figure}

\begin{definition}
\label{defn:ruling}
A \textbf{normal ruling} $\ruling$ of $\front$ is a bijection between the left and right cusps and, for each identified pair of cusps, two paths in $\front$ from the left cusp to the right. We require that: 
\begin{enumerate}
\item Two paths of $\ruling$ intersect only at cusps and crossings;
\item The two paths between the same two cusps are called \textbf{companions} of one another. Companion paths intersect only at the cusps; and
\item Two paths meeting at a crossing may pass through each other; see, for example, the bottom row of Figure~\ref{f:ruling-pieces}. Alternatively, two paths meeting at a crossing and their companion paths may be arranged as in the top row of Figure~\ref{f:ruling-pieces}; we call such a crossing a \textbf{switch}.
\end{enumerate}
\end{definition}

\begin{figure}[t]
\centering
\includegraphics[scale=.9]{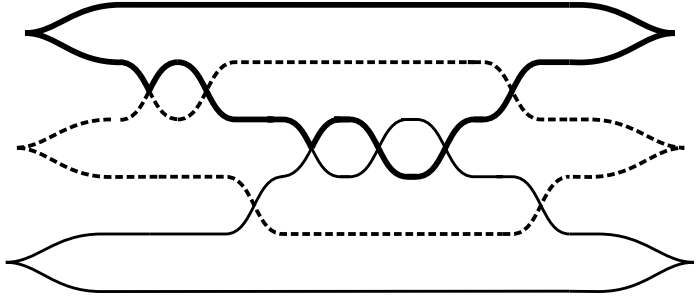}
\caption{A graded normal ruling with switches at the first and fourth crossings, returns at the sixth, seventh, and eighth crossings, and departures at the second, third, and fifth crossings.}
\label{f:ruling}
\end{figure}

A normal ruling is \textbf{graded} if all of its switched crossings are degree 0. A crossing is a \textbf{return} if the two paths meeting at the crossing and their companion paths are arranged as in one of the figures in the bottom row of Figure~\ref{f:ruling-pieces}. A return crossing is \textbf{graded} if it has degree 0. A crossing that is neither a switch nor a return is a \textbf{departure}; the arrangement of paths in a normal ruling near a departure can be seen by reflecting each figure in the bottom row of Figure~\ref{f:ruling-pieces} about a vertical axis. Figure~\ref{f:ruling} gives an example of a graded normal ruling.  An example of a graded normal ruling is given in Figure~\ref{f:ruling}. The set of graded normal rulings of $\front$ is $\mathcal{R}^0(\front)$. 

\subsubsection{2-graded Normal Rulings}
A normal ruling $\ruling$ of $\front$ is \textbf{2-graded} if all switched crossings of $\ruling$ are positive; see Figure~\ref{f:nearly-plat}~(b). The set of 2-graded normal rulings of $\front$ is $\mathcal{R}^2(\front)$ and the \textbf{2-graded normal ruling polynomial} of $\Leg$ is $$R_{\Leg}^2 (z) =  \sum_{\ruling \in \mathcal{R}^2(\front)} z^{j(\ruling)},$$ where $j(\ruling)=\#(\mbox{switches}) - \#(\mbox{right cusps})$. As is implied by the notation, the 2-graded normal ruling polynomial is a Legendrian isotopy invariant \cite{Chekanov2005}. Given $n \in \N$, we define $f_{\Leg}^n$ to be the number of graded rulings $\ruling \in \mathcal{R}^2(L)$ satisfying $j(\ruling) = n$. Note that $f_{\Leg}^n$ is the coefficient of $z^n$ in $R^2_L(z)$. We define $f_{\max}(\Leg)$ to be $f^l_{\Leg}$ where $l$ is $\max\{n \in \N: f_{\Leg}^n \neq 0\}$. Note that if a 2-graded normal ruling $\ruling$ contributes to the count $f_{\max}(\Leg)$, then $s(\ruling) \geq s(\ruling')$ for all $\ruling' \in \mathcal{R}^2(\Leg)$. We say such a $2$-graded normal ruling \textbf{maximizes the number of switches}.

By Theorem 4.1 of \cite{Rutherford}, $z\cdot R_{\Leg}^2(z)$ is the coefficient, as a polynomial in $z$, of $a^{-tb(\Leg) - 1}$ in the HOMFLY polynomial $P_K(a,z)$ of the smooth knot type $K$ of $\Leg$. Therefore, if two Legendrian knots $\Leg_1$ and $\Leg_2$ are smoothly isotopic to a knot $K$ and $\overline{tb}(K) = tb(\Leg_1) = tb(\Leg_2)$ holds, then $R_{\Leg_1}^2(z) = R_{\Leg_2}^2(z)$ and, thus, $f_{\max}(\Leg_1) = f_{\max}(\Leg_2)$ hold. For a smooth knot $K$, let $f_{\max}(K)$ be $f_{\max}(\Leg)$ for any maximal tb Legendrian representative $\Leg$ of $K$. By the previous discussion, $f_{\max}(K)$ is well-defined. 

The following two Propositions are used in the proof of Theorem~\ref{t:two-cusp}. Legendrian knots have a well-defined connect sum operation \cite{Etnyre2003}, which, in the case of front diagrams, is described in Figure~\ref{f:connectedsum}. Proposition~\ref{prop:fmax} follows from the observation that two 2-graded normal rulings that maximize the number of switches for $\Leg$ and $\Leg'$ individually form a 2-graded normal ruling that maximizes the number of switches for $\Leg \# \Leg'$ and this 2-graded normal ruling has a switch at the crossing in Figure~\ref{f:connectedsum}.

\begin{proposition}
\label{prop:fmax}
Given Legendrian knots $\Leg$ and $\Leg'$, $$f_{\max}(\Leg \# \Leg') = f_{\max}(\Leg) \cdot f_{\max}(\Leg').$$ 
\end{proposition}

\begin{figure}[t]
\labellist
\small\hair 2pt
\pinlabel {$\Leg$} [tl] at 13 43
\pinlabel {$\#$} [br] at 80 28
\pinlabel {$\Leg'$} [tl] at 112 43
\pinlabel {$=$} [tl] at 147 43
\pinlabel {$\Leg$} [tl] at 180 43
\pinlabel {$\Leg'$} [tl] at 280 43
\endlabellist
\centering
\includegraphics[scale=.8]{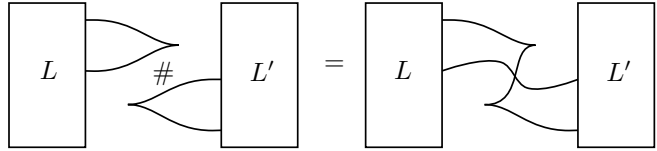}
\caption{Legendrian connect sum $\Leg \# \Leg'$.}
\label{f:connectedsum}
\end{figure}

\begin{proposition}
\label{prop:2bridgefmax}
If $K$ is a 2-bridge knot, then $f_{\max}(K) \in \{0,1\}$.
\end{proposition}

\begin{proof}
By definition, $f_{\max}(K)$ is $f_{\max}(\Leg)$ for any maximal tb Legendrian representative of $K$. By \cite{Ng2001a}, a 2-bridge knot $K$ has a maximal tb Legendrian representative $\Leg$ whose front diagram $\front$ has exactly four cusps. Therefore, it suffices to show $f_{\max}(\Leg) \in \{0,1\}$ for such a Legendrian knot. If $\Leg$ admits no more than one 2-graded normal ruling, then the claim obviously holds. 

\begin{figure}[t]
\centering
\includegraphics[scale=.9]{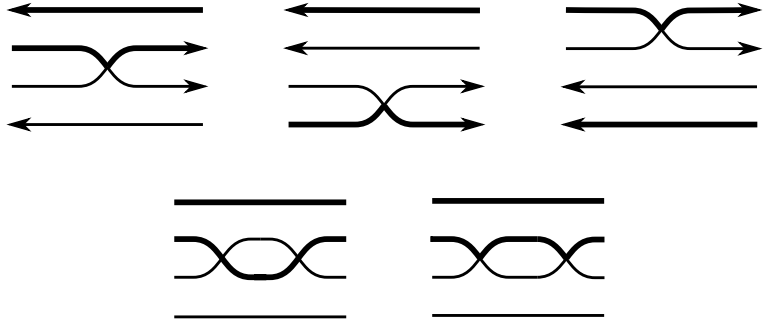}
\caption{(Top row) The orientation of companion paths near a switch of a 2-graded normal ruling. (Bottom row) An example of the 2-graded normal ruling $\ruling'$ (left) and $\ruling''$ (right) in the proof of Proposition~\ref{prop:2bridgefmax}}
\label{f:pos-neg-near-switches}
\end{figure}

Orient $\Leg$. In a 2-graded normal ruling, a switch may only occur at a positive crossing. Thus, in a 2-graded normal ruling $\ruling$, each path has a well-defined orientation and two paths that are companions of each another are oppositely oriented; see the top row of Figure~\ref{f:pos-neg-near-switches}. Label the crossings of the front diagram $\front$, from left to right, $q_1, \hdots, q_m$. Suppose $\ruling$ and $\ruling'$ are distinct 2-graded normal rulings of $\Leg$. We will show $\ruling$ and $\ruling'$ cannot both maximize the number of switches. The claim then follows directly. Since $\ruling \neq \ruling'$, there must exist a crossing that is a switch for one of $\ruling$ or $\ruling'$, but not for the other. Choose the left-most crossing $q_i$ for which this is the case and, without loss of generality, assume $q_i$ is a switch for $\ruling$, but not for $\ruling'$. Note that the paths of $\ruling$ and $\ruling'$ are identical to the left of $q_i$ and, therefore, $\ruling'$ has a departure at $q_i$ and, consequently, $q_i$ is not the right-most crossing of $\front$. The crossing $q_{i+1}$ must be a positive crossing, since $\ruling$ has a switch at $q_i$, companion paths cannot cross, and $\front$ has only 4 strands. Since $\ruling$ and $\ruling'$ agree to the left of $q_i$ and $\ruling'$ has a departure at $q_i$, $\ruling'$ must have a return at $q_{i+1}$. A new normal ruling $\ruling''$ may be constructed that agrees with $\ruling'$ to the left of $q_i$ and to the right of $q_{i+1}$ and has switches at $q_i$ and $q_{i+1}$; see, for example, the bottom row of Figure~\ref{f:pos-neg-near-switches}. Therefore, $\ruling'$ does not maximize the number of switched crossings. Therefore, $\front$ has a unique 2-graded normal ruling that maximizes the number of switches and $f_{\max} ( K ) = f_{\max}(\Leg) = 1$ holds.
\end{proof}

%% file: Sections/MCS.tex
\subsection{Morse Complex Sequences}
\label{ss:MCSs}

We again begin with a fixed Legendrian knot $\Leg$ with rotation number 0 and nearly plat front diagram $\front$ with fixed Maslov potential $\Maslov$. The most general definition of a Morse complex sequence, abbreviated MCS, is given in \cite{Henry2014}, however the relevant definition for this article is in \cite{Henry2013}. Regardless, we are able to avoid most technical details of MCSs and instead work with two special types of MCSs. 

A \textbf{handleslide} on $\front$ is a vertical line segment whose endpoints are on strands of $\front$ with the same Maslov potential and which does not intersect a crossing or cusp of $\front$; see the vertical line segments in Figure~\ref{f:marked}. Loosely speaking, an MCS consists of a collection of handleslides and a finite sequence of $\Z / 2 \Z$ chain complexes with consecutive chain complexes related by chain maps that depend on the crossings and cusps of $\front$ and the handleslides. The MCSs we consider have a special form that allows the sequence of chain complexes to be recovered from $\front$ and the collection of handleslides. Thus, for our purposes, Morse complex sequences will be defined by a collection of handleslides on $\front$. 

\begin{figure}[t]
\centering
\includegraphics[scale=.9]{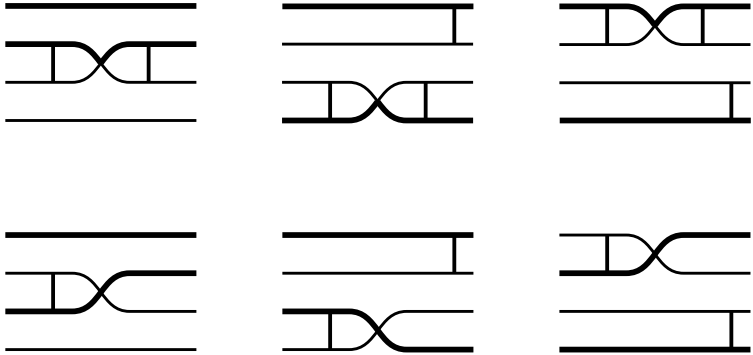}
\caption{The arrangement of handleslides near switches (top row) and graded returns (bottom row) in an SR-form MCS.}
\label{f:marked}
\end{figure}

\begin{definition}
\label{defn:SR-form}
An \textbf{SR-form Morse complex sequence} $\MCS$ of a front diagram $\front$ and graded normal ruling $\ruling$ consists of a collection of handleslides arranged as follows:

\begin{enumerate}
	\item Near each switched crossing $q$ of $\ruling$, handleslides of $\MCS$ are arranged as in the top row of Figure~\ref{f:marked};
	\item Let $R$ be a subset of the graded return crossings of $\ruling$. For each crossing $q$ in $R$, handleslides of $\MCS$ are arranged near $q$ as in the bottom row of Figure~\ref{f:marked}. We say $q$ is a \textbf{marked} graded return.  
\end{enumerate}
The set $MCS^{SR}(\front)$ consists of all SR-form MCSs of $\front$. A front diagram with an SR-form MCS and its associated graded normal ruling is given in Figure~\ref{f:SR-example}.
\end{definition}

\begin{figure}[t]
\centering
\includegraphics[scale=.9]{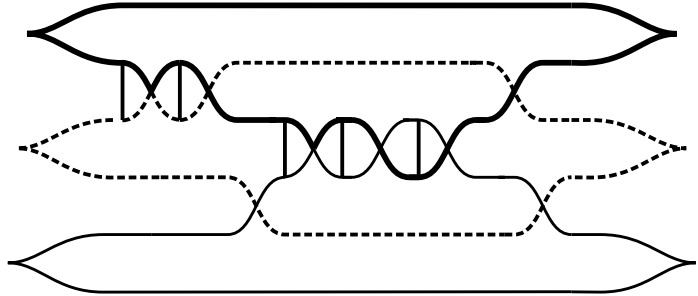}
\caption{A Morse complex sequence that is both SR-form and A-form.}
\label{f:SR-example}
\end{figure}

\begin{definition}
\label{defn:A-form}
An MSC $\MCS$ of $\front$ is an \textbf{A-form Morse complex sequence} if there exists a collection $\{p_1, \hdots, p_n\}$ of degree 0 crossings of $\front$ so that:

\begin{enumerate}
	\item For each $1 \leq i \leq n$, there exists a handleslide just to the left of $p_i$ with endpoints on the two strands of $\front$ that cross at $p_i$; and 
	\item $\MCS$ has no other handleslides. 
\end{enumerate}
The set $MCS^{A}(\front)$ consists of all A-form MCSs of $\front$. A front diagram with an A-form MCS is given in Figure~\ref{f:SR-example}.
\end{definition}

In \cite{Henry2013}, an MCS $\MCS$ is assigned a differential graded algebra $(\alg(\front), d^{\MCS})$. Note that $\alg(\front)$ is the same algebra as in the definition of the Chekanov-Eliashberg algebra given in Section~\ref{s:LCH}. Recall $A(\front)$ is the $\Z / 2 \Z$ vector space freely generated by labels assigned to the crossings and right cusps of $\front$. The restriction of $d$ to monomials of length 1 gives a map $d^{\MCS}_1 : A(\front) \to A(\front)$ and, in \cite{Henry2013}, it is shown that $d^{\MCS}_1 \circ d^{\MCS}_1 = 0$. Consequently, $(A(\front), d^{\MCS}_1)$ is a chain complex called the \textbf{linearization} of $(\alg(\front), d^{\MCS})$. We let $h^{\MCS}_i$ be the dimension in degree $i$ of the homology of $(A(\front), d^{\MCS}_1)$ and define the \textbf{MCS polynomial} to be $$P_{\MCS}(t) = \sum_{i \in \Z} h^{\MCS}_i t^i.$$ 

The proofs of Theorems~\ref{t:two-cusp}~and~\ref{t:infinite-family} require a careful analysis of the map $d^{\MCS}_1$ in the case that $\MCS$ is an SR-form MCS. Given generators $a$ and $b$ in $A(\front)$, the coefficient of $b$ in $d^{\MCS}_1 a$ is the mod 2 count of certain objects, called chord paths, originating at $a$ and terminating at $b$. Given an $x$-coordinate $x_0$ that is not the $x$-coordinate of any crossing, cusp, or handleslide, a \textbf{chord} $\lambda = (x_0, [i,j])$ is a vertical line segment with $x$-coordinate $x_0$ and endpoints on strands $i < j$ of \front, where the strands of $\front$ above $x = x_0$ are numbered $1, 2, \hdots$ from top to bottom. The following definition adapts Definition 5.1 of \cite{Henry2013} to the case that $\MCS$ is an SR-form MCS.

\begin{figure}[t]
\labellist
\small\hair 2pt
\pinlabel {(a)} [tl] at 17 88
\pinlabel {(b)} [tl] at 17 52
\pinlabel {(c)} [tl] at 17 0
\pinlabel {(d)} [tl] at 96 71
\pinlabel {(e)} [tl] at 169 71
\pinlabel {(f)} [tl] at 242 71
\pinlabel {(g)} [tl] at 313 71
\pinlabel {(h)} [tl] at 96 0
\pinlabel {(i)} [tl] at 169 0
\pinlabel {(j)} [tl] at 242 0
\pinlabel {(k)} [tl] at 313 0
\endlabellist
\centering
\includegraphics[scale=.9]{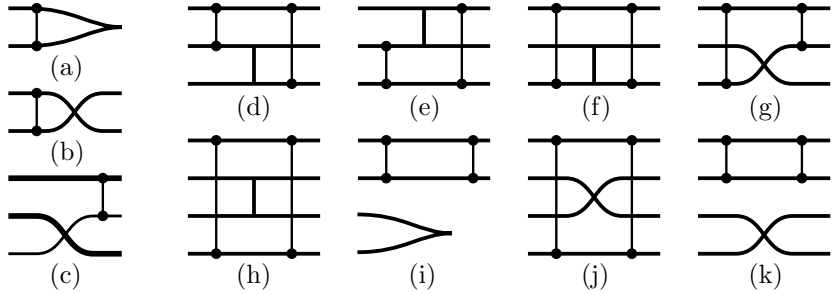}
\caption{Chords are represented pictorially as vertical lines with solid circles at the endpoints so as to distinguish them from handleslides. (a)-(b): The originating chord of a chord path. (c): An example of a terminating chord of a chord path. A second example can be seen by reflecting (c) across a horizontal axis. (d)-(k): Example behavior of consecutive chords in a chord path.}
\label{f:chord-path-fundamentals}
\end{figure}

\begin{definition}
\label{defn:chord-path}
Suppose $\MCS$ is an SR-form MCS of $\front$ with graded normal ruling $\ruling$. Suppose $a$ and $b$ are generators in $A(\front)$ with $|a| = |b| +1$ and the $x$-coordinate of $b$ is less than the $x$-coordinate of $a$. Let $x_{m+1} < x_{m} < \hdots < x_1$ be the $x$-coordinates of all crossings, cusps, and handleslides between $a$ and $b$, inclusive. In particular, $x_1$ (resp. $x_{m+1}$) is the $x$-coordinate of $a$ (resp. $b$). Choose $z_{m} < z_{m-1} < \hdots < z_{1}$ so that $x_{i+1} < z_i < x_{i}$ for all $1 \leq i \leq m$. A \textbf{chord path} from $a$ to $b$ is a finite sequence of chords $\Lambda = (\lambda_1, \hdots, \lambda_{m})$, so that: 
\begin{enumerate}
	\item For all $1 \leq i \leq m$, the $x$-coordinate of $\lambda_i$ is $z_i$;
	\item The endpoints of the chord $\lambda_1$ are on the two strands that cross at $a$, in the case that $a$ is a crossing, or on the two strands that terminate at $a$, in the case that $a$ is a right cusp; see Figure~\ref{f:chord-path-fundamentals} (a) and (b). We say $\Lambda$ \textbf{originates} at $a$;
	\item The formula $|a| = |b| +1$ holds, $\front$ is nearly plat, and all right cusps have degree 1. Consequently, $b$ must be a crossing. Number the strands of $\front$, from top to bottom,  $1, 2, \hdots$ just to the left of $b$ and suppose strands $k$ and $k+1$ cross at $b$. We require that, for the chord $\lambda_{m} = (z_{m},[i,j])$, either $j = k$ or $i=k+1$ hold, and the strands numbered $i$ and $j$ just to the \emph{left} of $b$ are companion paths of $\ruling$; see, for example, Figure~\ref{f:chord-path-fundamentals} (c). We say $\Lambda$ \textbf{terminates} at $b$.
	\item The endpoints of consecutive chords in a chord path satisfy conditions based on the crossing, cusp, or handleslide that appears between them.

\begin{enumerate}
	\item \textbf{Right Cusp or Crossing:} If a right cusp or crossing appears between $\lambda_i$ and $\lambda_{i+1}$, then the endpoints of $\lambda_i$ and $\lambda_{i+1}$ are on the same strands of $\front$; see, for example, Figure~\ref{f:chord-path-fundamentals} (g), (i), (j), and (k).
	\item \textbf{Left cusp:} Since $\front$ is nearly plat, the crossing $b$ is right of all left-cusps and so we need not consider this case. 
	\item \textbf{Handleslide:} Suppose a handleslide mark with endpoints on strands $k < l$ appears between chords  $\lambda_i = (z_i, [t_i, b_i])$ and $\lambda_i=(z_{i+1},[t_{i+1}, b_{i+1}])$. Then one of the following is satisfied:
		
\begin{enumerate}
	\item Equations $t_i = t_{i+1}$ and $b_i = b_{i+1}$ hold. See, for example, Figure~\ref{f:chord-path-fundamentals} (f) and (h); 
	\item Equations $b_i = l$, $b_{i+1} = k$, and $t_{i+1} = t_i$ and inequality $t_i < k$ hold. See Figure~\ref{f:chord-path-fundamentals} (d); 
	\item Equations $t_i = k$, $t_{i+1} = l$, and $b_{i+1} = b_i$ and inequality $b_i > l$ hold. See Figure~\ref{f:chord-path-fundamentals} (e).
\end{enumerate}
In the case of (ii) or (iii), we say the chord path \textbf{jumps along} the handleslide.
\end{enumerate}
	
\end{enumerate}

\end{definition}

In Figure~\ref{f:D_3-with-SRform}, a chord path is given that originates at $c_3$ and terminates at $a_3$. The set $\mmc^{\MCS}(a;b)$ consists of all chord paths originating at $a$ and terminating at $b$.

Given an SR-form MCS $\MCS$ of $\front$, the differential $d^{\MCS}_1$ of the linearization $(A(\front), d^{\MCS}_1)$ is defined on a generator $a$ of $A(\front)$ by $$d^{\MCS}_1 a = \sum \# \mmc^{\MCS}(a;b) b $$ where $\# \mmc^{\MCS}(a;b)$ is the mod 2 count of chord paths in $\mmc^{\MCS}(a;b)$ and the sum is over all generators $b$ of $A(\front)$.

\begin{remark}
\label{r:MCS-simplification}
Fix an SR-form MCS $\MCS$. Let $C_k \subset A(\front)$ be the $\Z / 2 \Z$ vector subspace generated by the set $\{q_j : |q_j|=k\}$ and $d^{\MCS}_{1,k} : C_k \to C_{k-1}$ be the restriction of the differential $d^{\MCS}_1$ to $C_k$. Let $n_k$ be the dimension of $C_k$ and $r^{\MCS}_k$ be the rank of $d^{\MCS}_{1,k}$. The coefficient $h^{\MCS}_k$ in $P_{\MCS}(t)$ is given by 
\begin{equation*}
h^{\MCS}_k = n_k - r^{\MCS}_{k} - r^{\MCS}_{k+1}.
\end{equation*}
\begin{enumerate}
	\item The Legendrian knots considered in Theorem~\ref{t:infinite-family} have the property that $n_k$ is $0$ if $|k| \geq 2$ holds. Therefore, $h^{\MCS}_k$ is $0$ if $|k| \geq 2$. Since $n_{-2}$ is $0$, $r_{-1}^{\MCS}$ is $0$ as well and $h_{-1} = n_{-1} - r_0^{\MCS}$ holds. By Corollary 7.12 of \cite{Henry2013}, 	
\begin{equation}
\label{eq:poly-set}
\{ P_{\MCS}(t) \}_{\MCS \in MCS(\front)}=\{ P_{\aug}(t) \}_{\aug \in Aug(\front)}
\end{equation}
holds, where $MCS(\front)$ is the set of MCSs of $\front$ as defined in Definition 4.2 of \cite{Henry2013}. Thus, the duality result, Theorem~\ref{t:duality}, applies to $(A(L), d^{\MCS}_1)$. Consequently, $h_1^{\MCS} = h_{-1}^{\MCS} + 1$ holds. The Thurston-Bennequin number $tb(\Leg)$ is computed by $P_{\aug}(-1)$, for any augmentation $\aug \in Aug(\front)$ and so Equation~\ref{eq:poly-set} implies $tb(\Leg)$ is $P_{\MCS}(-1)$. Thus, $h_0 = tb(\Leg) + 2 h_{-1}^{\MCS} + 1$ holds. Therefore, 
\begin{equation*}
P_{\MCS}(t) = (n_{-1} - r_0^{\MCS}) t^{-1} + (tb(\Leg) + 2 (n_{-1} - r_0^{\MCS}) + 1) + (n_{-1} - r_0^{\MCS}+1)t
\end{equation*}
Note that $P_{\MCS}(t)$ depends only on $tb(\Leg)$, $n_{-1}$ and $r^{\MCS}_0$ and only one of these three values, $r^{\MCS}_0$, depends on the SR-form MCS $\MCS$.  This observation will simplify the proof of Theorem~\ref{t:infinite-family}.
	\item The Legendrian knots considered in Theorem~\ref{t:two-cusp} have the property that $n_k$ is $0$ if $|k| \geq 3$.  An investigation similar to that above shows that $P_{\MCS}(t)$ depends only on $tb(\Leg)$, $n_{-1}$, $n_{-2}$, $r^{\MCS}_{0}$ and $r^{\MCS}_{-1}$ and only two of these values, $r^{\MCS}_0$ and $r^{\MCS}_{-1}$, depend on $\MCS$. This observation will simplify the proof of Theorem~\ref{t:two-cusp}.
\end{enumerate}
\end{remark}

%% file: Sections/Theorem1.tex
\section{The Main Results}
\label{ch:results}

\begin{theorem}
\label{t:two-cusp}
If the front diagram of a Legendrian knot has exactly four cusps, then the Legendrian knot has at most one Chekanov polynomial. 
\end{theorem}

\begin{proof}
Suppose the front diagram $\front$ of the Legendrian knot $\Leg$ has exactly four cusps. If $Aug(\front)$ is empty, then $\Leg$ has no Chekanov polynomials. Thus, we may assume $Aug(\front)$ is non-empty. We may assume $\front$ is nearly plat, since every Legendrian knot is equivalent to a Legendrian knot with such a front diagram by a Legendrian isotopy that does not change the number of cusps. Fix $\aug \in Aug(\front)$. We will show $P_{\aug}(t)$ is independent of $\aug$ and, thus, the result follows. 

First we make two observations that allow us to translate this problem from augmentations to SR-form MCSs. Combining Theorem 1.6 in \cite{Henry2011} and Theorem 5.5 of \cite{Henry2013}, $$\{ P_{\MCS}(t) \}_{\MCS \in MCS^{SR}(\front)} = \{ P_{\MCS}(t) \}_{\MCS \in MCS(\front)}$$ holds. By Corollary 7.12 of \cite{Henry2013}, $$\{ P_{\MCS}(t) \}_{\MCS \in MCS(\front)}=\{ P_{\aug}(t) \}_{\aug \in Aug(\front)}$$ holds. Therefore, $$\{ P_{\MCS}(t) \}_{\MCS \in MCS^{SR}(\front)} = \{ P_{\aug}(t) \}_{\aug \in Aug(\front)}$$ holds and so it suffices to show that, given any SR-form MCS $\MCS$, $P_{\MCS}(t)$ is independent of $\MCS$. 

Since $\front$ has exactly two left cusps, the dimension of $C_k$, also denoted $n_k$, is $0$ if $|k| \geq 3$. Therefore, by Remark~\ref{r:MCS-simplification} it suffices to show that $r^{\MCS}_0$ and $r^{\MCS}_{-1}$ are independent of $\MCS$. The differential $d^{\MCS}_{1,0}$ (resp. $d^{\MCS}_{1,-1}$) is determined by chord paths that originate at a crossing of degree $0$ (resp. degree $-1$) and terminate at a crossing of degree $-1$ (resp. degree $-2$). Let $\ruling$ be the graded normal ruling associated with $\MCS$ and fix a Maslov potential $\Maslov$ so that the smallest value assigned to a strand of $\front$ by $\Maslov$ is 0. Since a switch of a graded normal ruling occurs only at degree 0 crossings, every path in a graded normal ruling has a well-defined Maslov potential. Since two companion paths originate at a common left cusp, their Maslov potentials differ by 1 and the path with the larger $z$-coordinate has the larger Maslov potential. Note that if $\lambda_i = (z_i, [k,l])$ is a chord in a chord path originating from a degree $0$ (resp. degree $-1$) crossing, then $\Maslov(k) - \Maslov(l) = 0$ holds (resp. $\Maslov(k) - \Maslov(l) = -1$ holds). By looking at the arrangement of handleslide marks of $\MCS$ near switches and marked graded returns of $\ruling$ (see Figure~\ref{f:marked}) and the Maslov potentials of the paths of $\ruling$ near such handleslides, we see that it is not possible for a chord path originating from either a degree $0$ or $-1$ crossing to jump along a handleslide mark of $\MCS$. Therefore, $r^{\MCS}_0$ and $r^{\MCS}_{-1}$ do not depend on the handleslide marks of $\MCS$. 

We must also check that all graded normal rulings of $\front$ look identical near degree -1 and -2 crossings, since, by Definition~\ref{defn:chord-path} (3), the arrangement of a graded normal ruling near such crossings determines which chord paths can terminate at that crossing.  There are two possibilities for the Maslov potential values at the left cusps of $\front$.
\begin{figure}[t]
\labellist
\small\hair 2pt
\pinlabel {(a)} [tr] at 50 93
\pinlabel {1} [tr] at 0 150
\pinlabel {0} [tr] at 0 136
\pinlabel {1} [tr] at 0 120
\pinlabel {0} [tr] at 0 104

\pinlabel {(b)} [tr] at 167 93
\pinlabel {1} [tr] at 115 150
\pinlabel {2} [tr] at 115 136
\pinlabel {0} [tr] at 115 120
\pinlabel {1} [tr] at 115 104

\pinlabel {(c)} [tr] at 282 93
\pinlabel {1} [tr] at 230 150
\pinlabel {2} [tr] at 230 136
\pinlabel {0} [tr] at 230 120
\pinlabel {1} [tr] at 230 104

\pinlabel {(d)} [tr] at 50 -2
\pinlabel {1} [tr] at 0 56
\pinlabel {0} [tr] at 0 40
\pinlabel {2} [tr] at 0 24
\pinlabel {1} [tr] at 0 8

\pinlabel {(e)} [tr] at 167 -2
\pinlabel {2} [tr] at 115 56
\pinlabel {1} [tr] at 115 40
\pinlabel {0} [tr] at 115 24
\pinlabel {1} [tr] at 115 8

\pinlabel {(f)} [tr] at 282 -2
\pinlabel {1} [tr] at 230 56
\pinlabel {2} [tr] at 230 40
\pinlabel {1} [tr] at 230 24
\pinlabel {0} [tr] at 230 8

\endlabellist
\centering
\includegraphics[scale=.9]{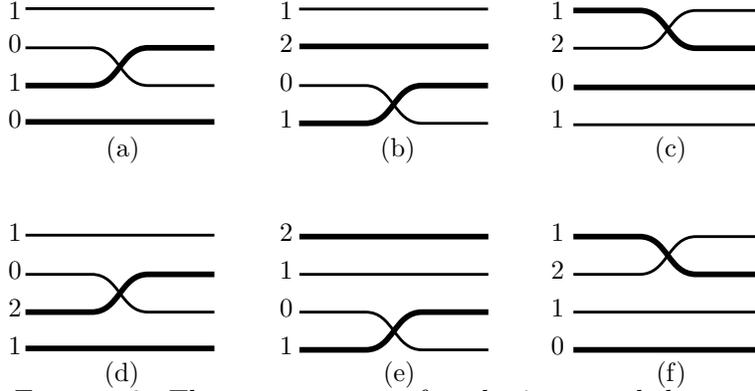}
\caption{The arrangement of paths in a graded normal ruling near a degree -1 or -2 crossing. The Maslov potential of each path is indicated.}
\label{f:neg-crossings}
\end{figure}
\begin{enumerate}
	\item The Maslov potential $\Maslov$ assigns the upper strands of both left cusps the number 1. As a consequence, there are no degree -2 crossings, since no two strands of $\front$ have Maslov potentials that differ by 2.  The conditions satisfied by the Maslov potentials of companion paths imply that, near a degree -1 crossing, the paths of a graded normal ruling must be arranged as in Figure~\ref{f:neg-crossings} (a).  
	\item The Maslov potential  $\Maslov$ assigns the upper strand of one left cusp the number 2 and the upper strand of the other left cusp the number 1. The conditions satisfied by the Maslov potentials of companion paths imply that, near a degree -2 (resp. -1) crossing, the paths of a graded normal ruling must be arranged as in Figure~\ref{f:neg-crossings} (d) (resp. Figure~\ref{f:neg-crossings} (b),(c),(e), or (f)).
\end{enumerate}
Note that in both cases, the arrangement of paths in a graded normal ruling near a degree -1 or -2 crossing depends only on $\Maslov$, which, in turn, depends only on $\front$. We can finally conclude that $r^{\MCS}_0$ and $r^{\MCS}_{-1}$ are independent of $\MCS$ and the result follows. 
\end{proof}

%% file: Sections/Theorem2.tex
\begin{figure}[t]
\centering
\includegraphics[scale=.75]{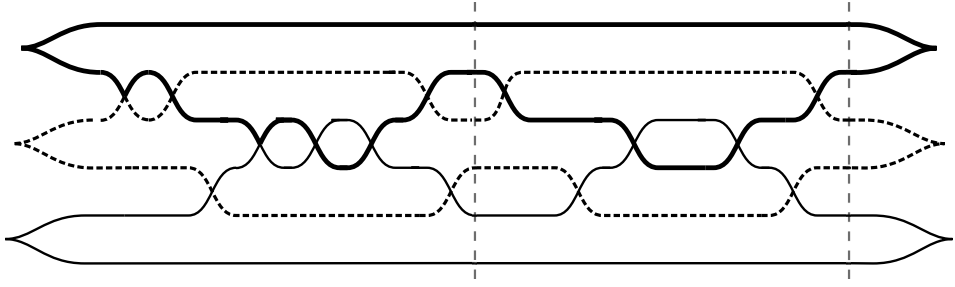}
\caption{The front diagram $\front_3$ in Theorem~\ref{t:infinite-family}. The crossings between the two vertical dotted lines are notated by $\omega = \sigma_2 \sigma_4 \sigma_3 \sigma_3 \sigma_4 \sigma_2$.}
\label{f:D_3}
\end{figure}

\begin{theorem}
\label{t:infinite-family}
For any natural number $m \geq 2$, there exists a front diagram $\front_m$ with graded normal ruling $\ruling$ and augmentations $\aug_1, \hdots, \aug_m$ so that:
\begin{enumerate}
	\item For all $1 \leq i \leq m$, $\Psi ( \aug_i) = \ruling$; 
	\item If $i \neq j$, then $P_{\aug_i}(t) \neq P_{\aug_j}(t)$; and
	\item The smooth knot type of $\front_m$ is prime.
\end{enumerate}\end{theorem}

\begin{proof}

For each $m \geq 2$, the front diagram $\front_m$ is nearly plat with 3 right cusps; the first such front diagram, $\front_2$, appears in Figure~\ref{f:nearly-plat}. Let $\sigma_i$ denote a crossing between strands $i$ and $i+1$, where the strands of $\front_m$ are numbered $1, 2, \hdots, 6$, from top to bottom, just to the left of the crossing. The front diagram $\front_2$ appears in Figure~\ref{f:nearly-plat} and, in sigma notation, is given by $\sigma_2 \sigma_2 \sigma_4 \sigma_3 \sigma_3 \sigma_3 \sigma_2 \sigma_4$. Let $ \omega$ be $\sigma_2 \sigma_4 \sigma_3 \sigma_3 \sigma_4 \sigma_2$. For $m>2$, the front diagram $\front_m$ is given by $\sigma_2 \sigma_2 \sigma_4 \sigma_3 \sigma_3 \sigma_3 \sigma_2 \sigma_4 (\omega)^{m-2}$. The front diagram $\front_3$ appears in Figure~\ref{f:D_3}. Let $\Leg_m$ be the Legendrian knot with front diagram $\front_m$ and $K_m$ be the smooth knot type of $\Leg_m$. 

Fix $m \geq 2$. We give names to those crossings in $\front_m$ that are important in the remainder of the proof. Beginning at the left-most crossing, label the first seven crossings $q_1, q_2, a_2, q_3, q_4, b_2$, and $c_2$; see Figure~\ref{f:D_3-with-SRform}. For $m>2$, label the second, fourth, and sixth crossings in the $i^{th}$ occurrence of $\omega$ in $\front_m$ by $a_i$, $b_i$, and $c_i$, respectively, see Figure~\ref{f:D_3-with-SRform}. 

The front diagram $\front_m$ has a graded normal ruling $\ruling$ with switches at $q_1$ and $q_4$ and graded returns at $b_2, \hdots, b_m$. For each $1 \leq i \leq m$, we define an SR-form MCS $\MCS_i$ of $\front$ and $\ruling$ as follows. Place handleslide marks around the switched crossings $q_1$ and $q_4$ as defined in Definition~\ref{defn:SR-form}; see the top left image in Figure~\ref{f:marked}. The set of marked graded returns $R_i$ of $\MCS_i$ is 
$$
R_i = \left\{ \begin{array}{rl}
 \{b_2\} &\mbox{ if $i=1$} \\
 \emptyset &\mbox{ if $i=2$} \\
  \{b_3, \hdots, b_i\} &\mbox{ if $3 \leq i \leq m$.}
       \end{array} \right.
$$
Place handleslide marks near each crossing in $R_i$ as in the bottom left image in Figure~\ref{f:marked}. 

By Definition~\ref{defn:A-form}, $\MCS_i$ is also an A-form MCS for each $1 \leq i \leq m$. By Theorem 5.2 of \cite{Henry2014}, there exists an explicit bijection $\Phi: MCS^A(\front) \to Aug(\front)$; for a crossing $q$, the A-form MCS $\MCS$ has a handleslide just to the left of $q$ between the strands crossing at $q$ if and only if $\Phi(\MCS) (q) =1$. Define $\aug_i$ to be $\Phi(\MCS_i)$. From the definition of the map $\Phi$, we see that the set of augmented crossings $A_i$ of $\aug_i$ is 
$$
A_i = \left\{ \begin{array}{rl}
 \{q_1, q_2, q_3, q_4, b_2\} &\mbox{ if $i=1$} \\
 \{q_1, q_2, q_3, q_4\} &\mbox{ if $i=2$} \\
 \{q_1, q_2, q_3, q_4, b_3, \hdots, b_i\} &\mbox{ if $3 \leq i \leq m$.}
       \end{array} \right.
$$
From the algorithmic construction of the map $\Psi: Aug(\front) \to R^{0}(\front)$ in \cite{Ng2006}, $\Psi(\aug_i)$ is $\ruling$ for all $1 \leq i \leq m$. From Theorem 7.3 of \cite{Henry2013}, $P_{\aug_i}(t) = P_{\MCS_i}(t)$ holds. Consequently, we may use the chord path approach given in Section~\ref{ss:MCSs} to compute $P_{\aug_i}(t)$. 

\begin{figure}[t]
\labellist
\small\hair 2pt
\pinlabel {$q_1$} [bl] at 55 100
\pinlabel {$q_2$} [bl] at 78 100
\pinlabel {$a_2$} [bl] at 97 51
\pinlabel {$q_3$} [bl] at 121 77
\pinlabel {$q_4$} [bl] at 150 77
\pinlabel {$b_2$} [bl] at 179 76
\pinlabel {$c_2$} [bl] at 205 100
\pinlabel {$a_3$} [bl] at 282 51
\pinlabel {$b_3$} [bl] at 363 76
\pinlabel {$c_3$} [bl] at 404 100
\endlabellist
\centering
\includegraphics[scale=.75]{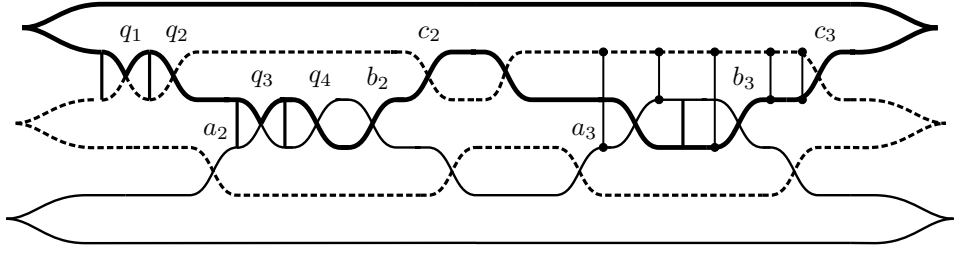}
\caption{The SR-form $\MCS_3$ on the front diagram $\front_3$, including the only chord path in the set $\mmc^{\MCS}(c_3;a_3)$, which determines the coefficient of $a_3$ in $d c_3$.}
\label{f:D_3-with-SRform}
\end{figure}

As noted in Remark~\ref{r:MCS-simplification}~(1), $P_{\MCS_i} (t)$ depends only on $tb(\Leg_m)$, $n_{-1}$, and $r^{\MCS_i}_{0}$. It is straightforward to verify that $tb(\Leg_m) = 2m - 3$. The set of degree -1 crossings is $\{a_2, \hdots, a_m\}$ and so $n_{-1} = m-1$ holds. An investigation of chord paths originating at a degree $0$ crossing reveals the following. If $i=1$, $d^{\MCS_i}_1 q=0$ holds for all degree 0 crossings and so $r^{\MCS_1}_0$ is $0$. If $2 \leq i \leq m$,  for a degree 0 crossing $q$ of $\front_m$, $d^{\MCS_i}_1 q=0$ holds, unless  $q \in \{c_2, \hdots, c_i\}$, in which case, if $2 \leq j \leq i$, then $d^{\MCS_i}_1 c_j = \sum_{k=2}^{j} a_k$. Figure~\ref{f:D_3-with-SRform} includes the only chord path in the set $\mmc^{\MCS}(c_3;a_3)$. Therefore, the rank of $d^{\MCS_i}_1$ is $i-1$ and so $r^{\MCS_i}_0 = i-1$ holds for all $1 \leq i \leq m$.

Since $tb(\Leg_m) = 2m - 3$, $n_{-1} = m-1$, and $r^{\MCS_i}_0 = i-1$ hold, Remark~\ref{r:MCS-simplification}~(1) and the formula $P_{\MCS_i} (t) = P_{\aug_i} (t)$ imply that 
\begin{equation}
\label{eq:MCS-polys}
P_{\aug_i} (t) = (m-i)t^{-1} + (4m-2i -2) + (m-i+1).
\end{equation} 
Therefore, for any $m \geq 2$ and $1 \leq i < j \leq m$, $P_{\aug_i}(t)$ is not equal to $P_{\aug_j}(t)$, but $\Psi(\aug_i) = \ruling$ for all $1 \leq i \leq m$.

We now show that $K_m$ is prime. Suppose, for contradiction, that $K_m$ is $K \# K'$. As $\front_m$ has exactly 3 left cusps, the bridge number of $K_m$, denoted $br(K_m)$, is at most 3. A classical result of Schubert \cite{Schubert} states $br(K_m) = br(K) + br(K') + 1$ holds. Therefore, $br(K)$ and $br(K')$ are both 2. By Theorem 3.4 of \cite{Etnyre2003}, there exist Legendrian knots $\Leg$ and $\Leg'$ smoothly isotopic to $K$ and $K'$, respectively, so that $\Leg_m = \Leg \# \Leg'$. If we connect sum $\Leg$ and $\Leg'$ as in Figure~\ref{f:connectedsum}, then it is clear that $\Leg_m$ admits a graded normal ruling if and only if $\Leg$ and $\Leg'$ both admit a graded normal ruling. A Legendrian knot admitting a graded normal ruling has maximal tb within its smooth knot class and rotation number $0$. Therefore, since $\Leg_m$ admits a graded normal ruling, $\Leg$ and $\Leg'$ do as well and both have maximal tb within their smooth knot class and rotation number 0. By Proposition~\ref{prop:2bridgefmax}, $f_{\max}(K)$ and $f_{\max}(K')$ are either 0 or 1 and, consequently, $f_{\max}(\Leg)$ and $f_{\max}(\Leg')$ are either 0 or 1. By Proposition~\ref{prop:fmax}, $f_{\max}(\Leg_m) = f_{\max}(\Leg) \cdot f_{\max}(\Leg')$ holds. Therefore, $f_{\max}(\Leg_m)$ is either 0 or 1. However, $f_{\max}(\Leg_m)$ is $2^{m+1}$, as we shall now demonstrate. It is easily verified that there are two 2-graded normal rulings of $\front_2$ that maximize the number of switches. The first has switches at crossings 1, 4, 5, and 6 and the second has switches at crossings, 1, 2, 5, and 7, where the crossings of $\front_m$ are numbered from left to right. Therefore,  $f_{\max}(\Leg_0)=2$. The block of crossings $\omega$ has exactly 2 possible 2-graded normal rulings and each has 2 switches. The first has switches at crossings 1 and 6 and the second has switches at crossings 3 and 4. Therefore, for $m>2$, $f_{\max}(\Leg_m) = 2 f_{\max}(\Leg_{m-1})$ holds. Therefore, $f_{\max}(\Leg_m)=2^{m+1}$ holds. We have arrived at a contradiction and, thus, $K_m$ must be prime. 
\end{proof}


\begin{remark}
Fix $m \geq 2$ and let $\Leg_m$ be the Legendrian knot from Theorem 2 whose front diagram is $\front_m$. Every Chekanov polynomial of $\Leg_m$ is achieved by an augmentation from the set $\{ \aug_1, \hdots, \aug_m\}$ from Theorem~\ref{t:infinite-family} as we shall now show. Let $\aug$ be an augmentation of $\front_m$ with Chekanov polynomial $P_{\aug}(t) = \sum_{i \in \Z} h_i t^i$. Let $n_i$ be the number of crossings of degree $i$ and $r_i$ be the rank of the map $ \df_{1,i+1}^{\aug}$. Then $h_i$ equals $n_i - r_i - r_{i-1}$. Note that $h_i$ is $0$ for all $|i| \geq 2$, since $n_i$ is $0$ for all $|i| \geq 2$. In addition, $tb(\Leg_m) = P_{\aug}(-1) = -h_{-1} + h_0 - h_1$ holds and, by Theorem~\ref{t:duality}, $h_{1} = h_{-1} + 1$. Therefore,  
\begin{align*}
h_0 &= tb(\Leg_m) + 2 h_{-1} + 1& &\\ \nonumber
&= w(\front_m) - 3 + 2 h_{-1} + 1& &[tb(\Leg_m)=w(\front_m)-3] \\ \nonumber
&= w(\front_m) - 3 + 2 ( n_{-1} - r_{-1} - r_{-2}) + 1& &[h_{-1} = n_{-1}-r_{-1}-r_{-2}] \\ \nonumber
&= w(\front_m) - 3 + 2 ( n_{-1} - r_{-1}) + 1& &[r_{-2}=0] \\ \nonumber
&= (n_0 - n_{-1} - n_{1}) - 3  + 2 ( n_{-1} - r_{-1}) + 1& &[w(\front_m)=n_0 - n_{-1} - n_{1}]\\ \nonumber
&= (n_0 - 2n_{-1}) + 2 ( n_{-1} - r_{-1}) -2& &[n_1=n_{-1}=m-1]\\ \nonumber
&= (4m-2 - 2n_{-1}) + 2 ( n_{-1} - r_{-1}) -2& &[n_0=4m-2]\\ \nonumber
&= 4m - 4 - 2 r_{-1}& & \nonumber 
\end{align*}

Since $n_{-1} = m-1$, $r_{-1}$ is in the set $\{0, \hdots, m-1\}$ and, thus, $h_0$ is in the set $\{4m-4, 4m-6, \hdots, 2m-2\}$. As a consequence of Theorem~\ref{t:duality}, the values $tb(\Leg_m)$ and $h_0$ determine $P_{\aug}(t)$. Every value in the set $\{4m-4, 4m-6, \hdots, 2m-2\}$ is achieved as $h_0$ for some augmentation in the statement of Theorem 2, as can be seen in Equation~\ref{eq:MCS-polys}. 
\end{remark}